\documentclass[conference]{IEEEtran}

\title{Delooping the sign homomorphism in univalent mathematics}

\author{%
  \IEEEauthorblockN{\'El\'eonore Mangel}
  \IEEEauthorblockA{\textit{\'Ecole Normale Sup\'erieure Paris-Saclay}\\ eleonore.mangel@ens-paris-saclay.fr}
  \and
  \IEEEauthorblockN{Egbert Rijke}
  \IEEEauthorblockA{\textit{University of Ljubljana}\\ egbert.rijke@fmf.uni-lj.si}
}

\usepackage{amsmath}
\usepackage{amsthm}
\usepackage{amssymb}
\usepackage{microtype}
\usepackage{booktabs}
\usepackage{xspace}
\usepackage{tikz-cd}
\usepackage{url}
\usepackage
  [ unicode=true,
  pdftitle={Delooping the sign homomorphism in univalent mathematics},
  pdfauthor={\'El\'eonore Mangel and Egbert Rijke},
  colorlinks=true,
  linkcolor=purple,
  urlcolor=blue,
  citecolor=blue]{hyperref}
\usepackage[capitalise,nameinlink]{cleveref}

\newtheorem{theorem}{Theorem}
\newtheorem{proposition}[theorem]{Proposition}
\newtheorem{lemma}[theorem]{Lemma}
\newtheorem{corollary}[theorem]{Corollary}

\theoremstyle{definition}
\newtheorem{definition}[theorem]{Definition}
\newtheorem{remark}[theorem]{Remark}

\mathchardef\dsh="2D
\newcommand{\blank}{\mathord{\hspace{1pt}\text{--}\hspace{1pt}}}

\newcommand{\agdaunimath}{agda-unimath\xspace}
\newcommand{\binomial}[2]{{#1 \choose #2}}
\newcommand{\brck}[1]{\|#1\|}
\newcommand{\Brck}[1]{\left\|#1\right\|}
\newcommand{\BS}{BS}
\newcommand{\BA}{BA}
\newcommand{\Card}[1]{\left|#1\right|}

\newcommand{\define}[1]{{\bf #1}}
\newcommand{\demb}{\mathbin{\hookrightarrow_{\mathsf{d}}}}
\newcommand{\dProp}{\mathsf{dProp}}
\newcommand{\evzero}{\mathsf{ev}_0}
\newcommand{\eqequiv}{\mathsf{eq\dsh{}equiv}}
\newcommand{\equiveq}{\mathsf{equiv\dsh{}eq}}
\newcommand{\fib}{\mathsf{fib}}
\newcommand{\Fin}[1]{[#1]}
\newcommand{\isfinite}{\mathsf{is\dsh{}finite}}
\newcommand{\N}{\mathbb{N}}
\newcommand{\Z}{\mathbb{Z}}
\newcommand{\F}{\mathbb{F}}
\newcommand{\refl}{\mathsf{refl}}
\newcommand{\sign}{\mathsf{sign}}
\newcommand{\suc}{\mathsf{succ}}
\newcommand{\swap}{\mathsf{swap}}
\newcommand{\unit}{\mathbf{1}}
\newcommand{\universecomponent}[1]{\UU_{(#1)}}
\newcommand{\UU}{\mathcal{U}}
\newcommand{\VV}{\mathcal{V}}
\newcommand{\id}{\mathsf{id}}
\newcommand{\transposition}[2]{(#1\,#2)}
\newcommand{\tr}{\mathsf{tr}}
\newcommand{\defeq}{:=}
\newcommand{\ap}{\mathsf{ap}}
\newcommand{\tot}{\mathsf{tot}}

\newcommand{\pr}{\mathsf{pr}}
\newcommand{\loopspace}{\Omega}
\newcommand{\Cycle}{\mathsf{Cycle}}
\newcommand{\cycledecomposition}{\mathsf{Cycle\dsh{}decomp}}
\newcommand{\EqReld}{\mathsf{Eq\dsh{}Rel_d}}
\newcommand{\Sigmadecomp}{\mathsf{\Sigma\dsh{}decomp}}
\newcommand{\cyc}{\mathsf{cycle}}
\newcommand{\lst}{\mathsf{list}}
\newcommand{\support}{\mathsf{support}}
\newcommand{\brckRezk}[1]{\brck{#1}_{\mathsf{Rezk}}}

\begin{document}

\maketitle

\begin{abstract}
  In univalent mathematics there are at least two equivalent ways to present the category of groups. Groups presented in their usual algebraic form are called abstract groups, and groups presented as pointed connected $1$-types are called concrete groups. Since these two descriptions of the category of groups are equivalent, we find that every algebraic group corresponds uniquely to a concrete group---its delooping---and that each abstract group homomorphisms corresponds uniquely to a pointed map between concrete groups.

  The $n$-th abstract symmetric group $S_n$ of all bijections $\Fin{n}\simeq\Fin{n}$, for instance, corresponds to the concrete group of all $n$-element types. The sign homomorphism from $S_n$ to $S_2$ should therefore correspond to a pointed map from the type $\BS_n$ of all $n$-element types to the type $\BS_2$ of all $2$-element types. Making use of the univalence axiom, we characterize precisely when a pointed map $\BS_n\to_\ast \BS_2$ is a delooping of the sign homomorphism. Then we proceed to give several constructions of the delooping of the sign homomorphism. Notably, the construction following a method of Cartier can be given without reference to the sign homomorphism. Our results are formalized in the \agdaunimath library.
\end{abstract}

\section{Introduction}

In this article we present several ways to construct a delooping of the sign homomorphism \cite{Bourbaki98}
\begin{equation*}
  \sign:S_n\to S_2
\end{equation*}
in univalent mathematics \cite{hottbook}. In homotopy type theory and univalent mathematics, a delooping of a group $G$ \cite{BvDR18} is a pointed connected $1$-type $BG$ equipped with a group isomorphism
\begin{equation*}
  G\cong\loopspace BG.
\end{equation*}
There are many ways of defining the delooping of a group. One way is via the Rezk completion of \cite{AKS2015}, and another way is via the type of $G$-torsors \cite{Symmetry}. Both constructions can be extended to also deloop group homomorphisms: The delooping of a group homomorphism $f:G\to H$ is a pointed map $Bf:BG\to_\ast BH$ equipped with a homotopy witnessing that the square
\begin{equation*}
  \begin{tikzcd}
    G \arrow[r,"f"] \arrow[d,swap,"\cong"] & H \arrow[d,"\cong"] \\
    \loopspace BG \arrow[r,swap,"\loopspace Bf"] & \loopspace BH
  \end{tikzcd}
\end{equation*}
commutes. Every group and every group homomorphism has a unique delooping. Indeed, the category of groups is equivalent to the category of pointed connected $1$-types \cite{Symmetry}.

For the symmetric group $S_n$, however, there is a very natural description of its delooping: By the univalence axiom it is simply the type $BS_n$ of all $n$-element types \cite{Yorgey14,FGGvdW18,hott-intro}. This suggests that there could be a similarly natural description of the delooping of the sign homomorphism. Recall that the sign homomorphism is the unique group homomorphism $S_n\to S_2$ that maps the transpositions $\transposition{i}{j}$ to the non-identity element of $S_2$ . It can be defined by
\begin{equation*}
  \sign(e) =
  \begin{cases}
    1 & \text{if the number of inversions of $e$ is even}\\
    -1 & \text{if the number of inversions of $e$ is odd,}
  \end{cases}
\end{equation*}
where an inversion of a permutation $e:\Fin{n}\simeq \Fin{n}$ is a pair $(i,j)$ of elements in $\Fin{n}$ such that $i<j$ and $e(i)>e(j)$. In other words, inversions of $e$ are strictly ordered pairs of which $e$ reverses the ordering.

Our goal in this article is therefore to construct a pointed map
\begin{equation*}
  Q:\BS_n\to_\ast \BS_2
\end{equation*}
that transforms an arbitrary $n$-element type into a $2$-element type in such a way that $Q$ deloops the sign homomorphism. The sign homomorphism is of course the only surjective group homomorphism $S_n\to S_2$. Another way of phrasing our goal is therefore that we will construct a pointed \emph{connected} map
\begin{equation*}
  \BS_n\to_\ast \BS_2.
\end{equation*}
We will show in \cref{theorem:recognition-delooping-sign-homomorphism} that the condition that $Q:\BS_n\to_\ast \BS_2$ is connected is indeed equivalent to the condition that $Q$ deloops the sign homomorphism.

However, since the delooping of the sign homomorphism can be readily constructed via the Rezk completion or via torsors, it is fair to ask why we would bother delooping the sign homomorphism in yet another way. One of our motivations is to better understand how to work with finite types in univalent mathematics, or perhaps more generally in any form of equivalence invariant mathematics. Perhaps our main motivation is simply that we want to see what classical topics of mathematics look like from a univalent point of view, and how to use this univalent point of view for programming with mathematical concepts in a computer proof assistant such as Agda.

In particular, since our goal is to construct a pointed map $\BS_n\to_\ast \BS_2$ that deloops the sign homomorphism, our constructions take arbitrary types $X$ of cardinality $n$ as input, i.e., not just the standard $n$-element type $\Fin{n}$. The type $\Fin{n}$ possesses a lot of extra structure that is not present on arbitrary $n$-element types, such as a total ordering. Indeed, when $n\geq 2$ it is not possible to choose a particular bijection from an $n$-element type $X$ to the standard $n$-element type $\Fin{n}$ since any such choice of bijection cannot itself be invariant under equivalences. For the same reason, it is not possible to choose an element of an $n$-element type when $n\geq 2$, even though finite types of cardinality $n\geq 2$ are clearly inhabited. While this may seem to limit the programs that we can write for arbitrary $n$-element types, we note that it is also precisely the equivalence invariant nature of our constructions that enables us to deloop the sign homomorphism as an operation from the type of $n$-element types to the type of $2$-element types.

\subsection*{Formalization in the agda-unimath library}

Our work has been formalized in the \agdaunimath library \cite{Agda-Unimath}. This is a new library of formalized mathematics from a univalent point of view, which is comparable in size to the agda standard library and the cubical agda library. The agda-unimath library contains a significant portion of finite mathematics, which made it useful for our project. We added the following files to the agda-unimath library for our project:

\begin{center}
  \begin{tabular}{|l|}
    \hline
    \href{https://unimath.github.io/agda-unimath/finite-group-theory.cartier-delooping-sign-homomorphism.html}{finite-group-theory.cartier-delooping-sign-homomorphism} \\
    \href{https://unimath.github.io/agda-unimath/finite-group-theory.delooping-sign-homomorphism.html}{finite-group-theory.delooping-sign-homomorphism} \\
    \href{https://unimath.github.io/agda-unimath/finite-group-theory.orbits-permutations.html}{finite-group-theory.orbits-permutations} \\
    \href{https://unimath.github.io/agda-unimath/finite-group-theory.permutations.html}{finite-group-theory.permutations} \\
    \href{https://unimath.github.io/agda-unimath/finite-group-theory.sign-homomorphism.html}{finite-group-theory.sign-homomorphism} \\
    \href{https://unimath.github.io/agda-unimath/finite-group-theory.simpson-delooping-sign-homomorphism.html}{finite-group-theory.simpson-delooping-sign-homomorphism} \\
    \href{https://unimath.github.io/agda-unimath/finite-group-theory.transpositions.html}{finite-group-theory.transpositions} \\
    \hline
  \end{tabular}
\end{center}

During the preparation of this article we have found some proofs that are quicker or conceptually clearer than the proofs we have formalised. While our main results are formalised, some of the proofs we present here do not correspond directly to the proofs that have been formalised.

A proof of the fact that the sign homomorphism is indeed a group homomorphism can be found in many standard texts on undergraduate algebra, such as \cite{rotman2012introduction,Bourbaki98}. We have formalized the proof of \cite{rotman2012introduction}, but we won't go into those details in this article.

\subsection*{Overview of the article}

In \cref{sec:preliminary} we establish some notation and recall some basic facts of univalent mathematics that we will use later in the article. In \cref{sec:rezk} we recall that a general way of delooping group homomorphisms exists via Rezk completions. In \cref{sec:finite-types} we recall the definition and basic properties of finite types, decidable subtypes of finite types, and equivalent ways of describing decidable equivalence relations on finite types.

Our original contributions start in \cref{sec:transpositions} and onwards. In \cref{theorem:cycle-decompositions} we show that the type of permutations on a finite set $X$ is equivalently described as the type of cycle decompositions of $X$. This fact is proven using the univalence axiom, and we use it to give a univalent perspective on the fact that the transpositions generate the symmetric groups.

In \cref{theorem:recognition-delooping-sign-homomorphism} of \cref{sec:detection-sign-homomorphism} we give five conditions on a pointed map $Q:\BS_n\to_\ast \BS_2$ that are equivalent to $Q$ being a delooping of the sign homomorphism. We then proceed to give four ways of delooping the sign homomorphism. In \cref{sec:delooping-sign-fixed-points} we use the fixed points of a certain family of $S_n$-actions to deloop the sign homomorphism. In \cref{sec:delooping-sign-orbits} we use the orbits of a similar family of $S_n$-actions to define a second delooping. \cref{cor:quotient-construction} of \cref{theorem:recognition-delooping-sign-homomorphism} is used in \cref{theorem:simpson} of \cref{sec:simpson} to construct a delooping of the sign homomorphism following a suggestion of Simpson. Finally, in \cref{theorem:cartier-delooping} of \cref{sec:cartier} we use \cref{cor:quotient-construction} again to construct a delooping of the sign homomorphism following a method of Cartier. This last method is especially noteworthy, since it is used to define a map $\BS_n\to_\ast \BS_2$ without referring to the sign homomorphism. We conclude the article in \cref{sec:alternating-groups}, where we construct the delooping of the alternating groups $A_n$.

\subsection*{Acknowledgements}

The present work contains the results of an internship project of the first-listed author, advised by the second-listed author, as part of the masters program at the \'Ecole Normale Sup\'erieure Paris-Saclay. The project topic was selected in the fall of 2021, and the project ran from February to July of 2022. During our project, a question on the well-definedness of the sign homomorphism became popular on Mathoverflow.net \cite{mathoverflow}. While that question did not directly concern the delooping of the sign homomorphism, we have benefited from an answer by Bjorn Poonen to this question. We would also like to thank Andrej Bauer, Alex Simpson, Ulrik Buchholtz, and Tom de Jong for helpful discussions during the project.

\subsection*{Grant Acknowledgements}

The second-listed author gratefully acknowledges support by the Air Force Office of Scientific Research through grant FA9550-21-1-0024, and support by the Slovenian Research Agency research program P1-0294.

\section{Preliminary definitions}\label{sec:preliminary}

In this article we will use the notation from \cite{hott-intro}. In particular, we will make extensive use of a universe $\UU$ of types, which is closed under dependent function types ($\Pi$), dependent pair types ($\Sigma$), identity types ($=$), coproduct types ($+$) and product types ($\times$), and propositional truncations. The elements of identity types are called \textbf{identifications}. Recall that the \textbf{univalence axiom} asserts that the canonical map
\begin{equation*}
  \equiveq:(A=B)\to (A\simeq B)
\end{equation*}
defined by $\equiveq(\refl)\defeq \id$ is an equivalence for any two types $A,B:\UU$, where the type $A\simeq B$ of equivalences from $A$ to $B$ is defined to be the type of bi-invertible maps from $A$ to $B$. For an extensive reference on equivalences and the univalence axiom, see \cite{hottbook}.

A type $P$ is said to be a \textbf{proposition} if it contains at most one element. A \textbf{decidable proposition} is a proposition $P$ such that $P+\neg P$ holds. We will write $\dProp$ for the type of decidable propositions. \textbf{Subtypes} of a type $X$ are simply families of propositions over $X$, and \textbf{decidable subtypes} of $X$ are families of decidable propositions over $X$.

\subsection*{The action on equivalences of type-indexed families of types}

One consequence of the univalence axiom which is of particular interest to us in the present work, is that subtypes of the universe are automatically closed under equivalences. Given a subtype $P$ of the universe $\UU$ and two types $X,Y:\UU$, we have a map
\begin{equation*}
  (X\simeq Y)\to (P(X)\Leftrightarrow P(Y)).
\end{equation*}
Furthermore, if we write $\UU_P$ for the type $\sum_{(X:\UU)}P(X)$, it follows from univalence that
\begin{equation*}
  (X=_{\UU_P} Y)\simeq (X\simeq Y).
\end{equation*}
These observations imply that any family $A:\UU_P\to\VV$ of types indexed by types in $P$ has an \textbf{action on equivalences}. That is, for any two types $X$ and $Y$ in $\UU_P$, there is a unique map $e\mapsto A_e$ such that the triangle
\begin{equation*}
  \begin{tikzcd}
    (X=Y) \arrow[dr,"\tr_A"] \arrow[d,swap,"\equiveq"] \\
    (X\simeq Y) \arrow[r,dashed,swap,"e\mapsto A_e"] & (A_X\simeq A_Y)
  \end{tikzcd}
\end{equation*}
commutes. This follows at once from the fact that $\equiveq$ is an equivalence by the univalence axiom. In other words, the evaluation map
\begin{equation*}
  \left(\prod_{(Y:\UU)}\prod_{(q:P(Y))}(X\simeq Y) \to (A_X\simeq A_Y)\right)\to (A_X\simeq A_X)
\end{equation*}
at the identity function is an equivalence for every type $X:\UU_P$. Specifically, the action $e\mapsto A_e$ of $A$ on equivalences is the unique family of maps such that $A_{\id}=\id$. The action of $A$ on equivalences often allows us to bypass or get a better handle on transport in $A$ along identifications that correspond via univalence to equivalences.

We also note that for any family of maps $f_X:A_X\to B_X$ indexed by $X:\UU_P$ and every equivalence $e:X\simeq Y$ in the subuniverse $P$ we get a commuting square
\begin{equation*}
  \begin{tikzcd}
    A_X \arrow[r,"A_e"] \arrow[d,swap,"f_X"] & A_Y \arrow[d,"f_Y"] \\
    B_X \arrow[r,swap,"B_e"] & B_Y.
  \end{tikzcd}
\end{equation*}
This follows by equivalence induction, a consequence of univalence, by the fact that $A_\id=\id$ and $B_\id=\id$.

\subsection*{Propositional truncations}

We briefly recall the propositional truncation operation, which is essential in our work. \textbf{Propositional truncation} is an operation $A\mapsto \brck{A}$ that takes an arbitrary type $A$ and returns a proposition $\brck{A}$. The proposition $\brck{A}$ comes equipped with a map $\eta:A\to\brck{A}$, which is called the \textbf{unit} of the propositional truncation, and it satisfies the universal property that any map $A\to P$ from $A$ into a proposition $P$ factors uniquely through $\eta$, as indicated in the diagram
\begin{equation*}
  \begin{tikzcd}
    A \arrow[dr] \arrow[d,swap,"\eta"] \\
    \brck{A} \arrow[r,dashed] & P.
  \end{tikzcd}
\end{equation*}
A good way of thinking about the proposition $\brck{A}$ is that it asserts that the type $A$ is \emph{inhabited}. In other words, propositional truncations give us a way to prove or assert that a type is inhabited without exhibiting a specified element of that type. Indeed, sometimes it is possible to show that a type is inhabited while it would be impossible to construct an element.

\section{Delooping group homomorphisms via the Rezk completion}\label{sec:rezk}

A basic way of delooping a group $G$ is to take the Rezk completion of the pregroupoid with one object $\ast$, and morphisms $\hom(\ast,\ast)\defeq G$. Recall from \cite{AKS2015} that a \textbf{precategory} $C$ consists of a type of objects, and for any two objects $X$ and $Y$ a set of morphisms $\hom(X,Y)$ equipped with composition and identities satisfying the laws of a category. A \textbf{category} is a precategory satisfying the `external univalence' condition that for any two objects $X$ and $Y$, the canonical map
\begin{equation*}
  (X=Y)\to (X\cong Y)
\end{equation*}
from identifications of $X$ to $Y$ to isomorphisms from $X$ to $Y$, is an equivalence. The \textbf{Rezk completion} of a precategory $C$ is a category $\brckRezk{C}$ equipped with a functor $\eta:C\to\brckRezk{C}$ satisfying the universal property that for any category $D$, we have an equivalence
\begin{equation*}
  \blank\circ\eta:(\brckRezk{C}\to D)\simeq (C\to D)
\end{equation*}
of categories. It is immediate that the Rezk completion is functorial, in the sense that for any functor $F:C\to D$ between precategories, there is a unique functor $\brckRezk{F}:\brckRezk{C}\to\brckRezk{D}$ such that the square of functors
\begin{equation*}
  \begin{tikzcd}
    C \arrow[r,"F"] \arrow[d,swap,"\eta"] & D \arrow[d,"\eta"] \\
    \brckRezk{C} \arrow[r,swap,"\brckRezk{F}"] & \brckRezk{D}
  \end{tikzcd}
\end{equation*}
commutes.

A \textbf{pregroupoid} is simply a precategory in which every morphism is an isomorphism and a \textbf{groupoid} is a Rezk complete pregroupoid. Note that groupoids are just $1$-types, and functors between groupoids are just maps between $1$-types. In particular, if $f:G\to H$ is a group homomorphism, then we may consider $f$ to be a functor from the one-object pregroupoid $G$ to the one-object pregroupoid $H$ and Rezk complete $f$. The Rezk completion of a group $G$ is a pointed connected $1$-type $BG$, and the Rezk completion of a group homomorphism $f:G\to H$ is a pointed map $Bf:BG\to_\ast BH$. The fact that $Bf$ is a delooping of the group homomorphism $f$ follows at once from the commuting square of functors
\begin{equation*}
  \begin{tikzcd}
    G \arrow[r,"f"] \arrow[d,swap,"\eta"] & H \arrow[d,"\eta"] \\
    BG \arrow[r,swap,"Bf"] & BH.
  \end{tikzcd}
\end{equation*}
The functor $G\mapsto BG$ via the Rezk completion is an equivalence from the category of groups to the category of pointed connected $1$-types. In particular, the delooping of the sign homomorphism exists. However, the delooping of the sign homomorphism via the Rezk completion is not described in terms of a `natural' operation that turns an arbitrary $n$-element type into a $2$-element type in a way that matches the sign homomorphism.

\section{Finite types in univalent mathematics}\label{sec:finite-types}

In this section we recall the definition of finite types, as they are commonly considered in univalent mathematics \cite{Yorgey14,FGGvdW18,hott-intro}. We will omit most proofs of our claims, but we note that they have been formalised.

\subsection*{Finite types}
The \textbf{standard finite types} are a type family $\Fin{\blank} : \N \to \UU$, which is defined recursively by $\Fin{0}:=\emptyset$ and $\Fin{k+1}:=\Fin{k}+\unit$. 
A type $X$ is said to be \textbf{finite} if it is merely equivalent to a standard finite type, i.e., if it comes equipped with an element of type
\begin{equation*}
  {\textstyle \isfinite(X):=\Brck{\sum_{(k:\N)}\Fin{k}\simeq X}}
\end{equation*}
The type $\isfinite(X)$ is equivalent to $\sum_{(k:\N)}\brck{\Fin{k}\simeq X}$. In particular, every finite type has a unique and well-defined cardinality of type $\N$.

A type $X$ is said to be \textbf{discrete} if its identity types are decidable, i.e., if there is an element
\begin{equation*}
  d(x,y):(x=y)+(x\neq y)
\end{equation*}
for each $x,y:X$. Hedberg's theorem \cite{hedberg1998coherence} asserts that discrete types are always sets in the sense that its identity types are propositions. This implies that being discrete is a property, which in turn implies that all finite types are discrete.

\begin{definition}
  The \textbf{type of $n$-element types} is defined as
  \begin{equation*}
    \BS_n := \sum_{X:\UU}\|\Fin{n}\simeq X\|.
  \end{equation*}
\end{definition}

A good way to think about $\BS_n$ is as the \emph{groupoid of all $n$-element sets}. The following proposition is an immediate consequence of the univalence axiom. Note that it implies that $\BS_n$ is not contractible for $n\geq 2$.

\begin{proposition}
  The type $\BS_n$ is a pointed connected type with loop space $S_n:= (\Fin{n}\simeq \Fin{n})$. 
\end{proposition}

On the other hand, the type $\sum_{(X:\UU)}\Fin{n}\simeq X$ is contractible by the univalence axiom. This shows the importance of expressing that $X$ is an $n$-element type using the propositional truncation, and not by imposing an equivalence $\Fin{n}\simeq X$ as structure on $X$.

The following proposition shows that a pointed $(n+1)$-element type is equivalently described as an $n$-element type. The underlying map of this equivalence maps an $(n+1)$-element type $X$ equipped with $x:X$ to the type
\begin{equation*}
  Y\defeq\sum\nolimits_{(y:X)}x\neq y,
\end{equation*}
and the underlying map of the inverse is given by $Y\mapsto Y+\unit$.

\begin{proposition}
  We have an equivalence
  \begin{equation*}
    {\textstyle\left(\sum_{(X:\BS_{n+1})}X\right)\simeq \BS_n}
  \end{equation*}
  from the type of pointed $(n+1)$-element types to the type of $n$-element types.
\end{proposition}

In the special case where $n=2$ we obtain the following corollary, which was already observed in Theorem II.2 of \cite{BR2017realprojective}.

\begin{corollary}
  The type of pointed $2$-element types  is contractible. Consequently, the evaluation map
  \begin{equation*}
    \evzero:(\Fin{2}\simeq X)\to X
  \end{equation*}
  given by $\evzero(e):= e(0)$ is an equivalence.
\end{corollary}

\subsection*{Decidable subtypes of finite types}
Next, we review the theory of decidable subtypes of finite types. Recall that a \define{decidable embedding} $B \demb A$ is an embedding $f : B\hookrightarrow A$ such that the fibers of $f$ are decidable, i.e., $f$ comes equipped with a dependent function
\begin{equation*}
  \prod_{(x:X)}\fib_f(x)+\neg\fib_f(x)
\end{equation*}
By theorem 17.4.2 of \cite{hott-intro} it follows that the type of decidable subtypes of a type $A$ is equivalent to the type
\begin{equation*}
  \sum_{(X:\UU)}X\demb A
\end{equation*}
of decidable embeddings into $A$. For the following definition, recall that if $B:\UU$ is a type, then we write
\begin{equation*}
  \UU_{(B)}\defeq \sum_{(X:\UU)}\brck{B\simeq X}.
\end{equation*}
In other words, the type $\UU_{(B)}$ is the \emph{connected component} of the universe at $B$.

\begin{definition}
  For any two types $A$ and $B$, we define the \define{binomial type}
  \begin{equation*}
    \binomial{A}{B}:=\sum_{Y:\universecomponent{B}}Y\demb A
  \end{equation*}
\end{definition}

In other words, the binomial type $\binomial{A}{B}$ is the type of decidable subtypes of $A$ of the same size as $B$. In particular, if $A$ is an $n$-element type and $B$ is an $k$-element type, then the type $\binomial{A}{B}$ of $k$-element subtypes of $A$ is an $\binomial{n}{k}$-element type.

\begin{proposition}[Theorem 17.6.9 of \cite{hott-intro}]
  If $A$ and $B$ are finite types with $n$ and $k$ elements respectively, then the type $\binomial{A}{B}$ is finite with $\binomial{n}{k}$ elements.
\end{proposition}

The type $\binomial{X}{\Fin{2}}$ of decidable $2$-element subtypes of $X$ is of particular interest to us. Note that
\begin{equation*}
  \binomial{X}{\Fin{2}}\simeq \sum_{(P:X\to\dProp)}{\textstyle \Brck{\Fin{2}\simeq\sum_{(x:X)}P(x)}}.
\end{equation*}

\subsection*{Decidable equivalence relations on finite types}

\begin{definition}
  A \textbf{decidable equivalence relation} on a type $X$ is an equivalence relation $R$ on $X$ such that $R(x,y)$ is a decidable proposition for each $x,y:X$. We will write $\EqReld(X)$ for the type of all decidable equivalence relations on $X$.
\end{definition}

The quotient of a finite set by an equivalence relation $R$ is finite if and only if $R$ is a decidable equivalence relation. Indeed, a minor modification of Theorem 18.2.5 in \cite{hott-intro} shows that we have an equivalence
\begin{equation*}
  \EqReld(X)\simeq\sum_{(Y:\F)}X\twoheadrightarrow Y.
\end{equation*}
for any finite set $X$. We will also use the following fact, which is stated in more general form in Exercise 18.4 of \cite{hott-intro}, and which also has been formalized:

\begin{proposition}\label{finitely-indexed-sigma-decomposition}
  The type $\EqReld(X)$ of decidable equivalence relations on a finite set $X$ is equivalent to the type
  \begin{equation*}
    \Sigmadecomp_\F(X)\defeq \sum_{(Y:\F)}\sum_{(Z:Y\to \sum_{(Z:\F)}\|Z\|)}{\textstyle X\simeq \sum_{(y:Y)}Z(y)}
  \end{equation*}
  of \textbf{finitely indexed $\Sigma$-decompositions} of $X$.
\end{proposition}

\section{Transpositions}\label{sec:transpositions}

\begin{definition}
  For any $2$-element type $X$ we define the automorphism $\swap_X:X\simeq X$ as the composite of the equivalences
  \begin{equation*}
    \begin{tikzcd}
      X \arrow[r,"\evzero^{-1}"] & (\Fin{2}\simeq X) \arrow[r,"\suc^{*}"] & (\Fin{2}\simeq X) \arrow[r,"\evzero"] & X
    \end{tikzcd}
  \end{equation*}
  where $\suc:\Fin{2}\simeq\Fin{2}$ is the successor equivalence on $\Fin{2}$.
\end{definition}

The swap function on a $2$-element set is not homotopic to the identity function, and it is idempotent. The swap function allows us to define transpositions on an arbitrary type.

\begin{definition}
  An equivalence $e:X\simeq X$ on a discrete type $X$ is said to be a \textbf{transposition} on $X$ if the type
\begin{equation*}
  \support(e)\defeq \sum_{(x:X)}e(x)\neq x
\end{equation*}
is a $2$-element decidable subtype of $X$.
\end{definition}

\begin{proposition}
  Every $2$-element decidable subtype of a discrete type $X$ uniquely determines a transposition.
\end{proposition}

\begin{proof}
  Consider a $2$-element decidable subtype $P$ of $X$. Then we define $f:X\to X$ as follows:
  \begin{equation*}
    f(x)=
    \begin{cases}
      \pr_1(\swap(f(x),p)) & \text{if }p:P(x)\\
      x & \text{otherwise}.
    \end{cases}
  \end{equation*}
  Then it follows that the support of $f$ is equivalent to the $2$-element decidable subtype $P$.

  To show uniqueness, assume that $f$ and $g$ are two transpositions with support $P$. Since $X$ is assumed to be a discrete type, it follows that $f(x)=g(x)$ on the complement of $P$. Furthermore, since $P$ is a $2$-element subtype of $X$ and both $f(x)\neq x$ and $g(x)\neq x$ on $P$, it follows that $f(x)=g(x)$ on $P$. We conclude that $f=g$.
\end{proof}

Our next goal is to prove that the transpositions generate the symmetric group $S_n$. Of course, this fact can be proven in the usual algebraic way, as we have also done in our formalization in agda-unimath. There is, however, a nice univalent perspective on the fact that every permutation is uniquely up to reordering a product of cyclic permutations. 

\begin{definition}
  Consider a finite type $X$ of cardinality $n$. The type of \textbf{cyclic structures} on $X$ is defined by
  \begin{equation*}
    \cyc(X)\defeq \sum_{(f:X\to X)}\brck{(\Z/n,\suc)=(X,f)}
  \end{equation*}
  The type of \textbf{finite cycles} is defined to be
  \begin{equation*}
    \Cycle_{\F} \defeq \sum_{(X:\F)}\cyc(X).
  \end{equation*}
  If $(X,f)$ is a finite cycle, we will refer to $X$ as the \textbf{underlying type}.
\end{definition}

We recall from \cite{Symmetry} that the type
\begin{equation*}
  \Cycle_n\defeq\sum_{(X:\BS_n)}\cyc(X)
\end{equation*}
is a delooping of the cyclic group $\Z/n$. Note that $\Z/0=\Z$, so there is no finite cycle of order $0$. In particular, every finite cycle has an inhabited underlying type. More generally, a cycle can be defined to be a set bundle over the circle with connected total space, but we won't need this level of generality here.

\begin{definition}
  Consider a finite set $X$. A \textbf{cycle decomposition} of $X$ is a triple $(Y,C,e)$ consisting of a set $Y$ equipped with a family $C:Y\to \Cycle_\F$ of finite cycles and an equivalence
  \begin{equation*}
    e:X\simeq \sum_{y:Y}C(y).
  \end{equation*}
  We will write $\cycledecomposition(X)$ for the type of cycle decompositions of $X$.
\end{definition}

\begin{lemma}\label{lemma:cyclic-test}
  Consider a finite type $X$ of cardinality $n$, and consider a map $f:X\to X$. Then the proposition
  \begin{equation*}
    \brck{(\Z/n,\suc)=(X,f)}
  \end{equation*}
  holds if and only if for every $x,y:X$ there is a natural number $k$ such that
  \begin{equation*}
    f^k(x)=y.
  \end{equation*}
\end{lemma}

\begin{proof}
  In the forward direction, we may assume $(\Z/n,\suc)=(X,f)$. The fact that for any two integers $x$ and $y$ modulo $n$ there is a natural number $k$ such that
  \begin{equation*}
    \suc^k(x)=y
  \end{equation*}
  modulo $n$ is an elementary consequence of the fact that $1$ is a generator of $\Z/n$.

  For the converse direction, note that $X$ is inhabited, so we may assume an element $x:X$. Since $X$ has cardinality $n$, it follows that for every $y:X$ there is a unique natural number $0\leq k < n$ such that $f^k(x)=y$. This establishes a bijection $e:\Z/n\equiv X$ such that $e(k)=f^k(x)$. It follows that the square
  \begin{equation*}
    \begin{tikzcd}
      \Z/n \arrow[d,swap,"\suc"] \arrow[r,"e"] & X \arrow[d,"f"] \\
      \Z/n \arrow[r,swap,"e"] & X
    \end{tikzcd}
  \end{equation*}
  commutes, and hence we obtain an identification $(\Z/n,\suc)=(X,f)$.
\end{proof}

In the following theorem we will show via a fairly lengthy proof that the type of permutations on $X$ is equivalent to the type of cyclic decompositions of $X$. Since it takes some work to prove this claim, it is fair to ask what the benefit is of this univalent approach. One reason we think it is of interest is that in contrast with the usual algebraic approach, where each permutation is shown to be a composite of cyclic permutations and this composite is unique up to reordering, is that here we assert directly that the type of permutations on $X$ is a moduli space of decompositions of $X$ into cyclic types. Whereas in the traditional approach one needs to take a quotient that identifies factorizations that are the same up to reordering, here we will instead find an application of the univalence axiom. Our secondary motivation is that we are interested in seeing what mathematics from a univalent perspective looks like, and we find this a striking example.

\begin{theorem}\label{theorem:cycle-decompositions}
  Consider a finite set $X$. Then we have an equivalence
  \begin{equation*}
    \cycledecomposition(X)\simeq (X\simeq X).
  \end{equation*}
\end{theorem}

\begin{proof}
  First, we note that by rearranging the data of cycle decompositions, we immediately obtain an equivalence
  \begin{equation*}
    \cycledecomposition(X)\simeq\sum_{((Y,Z,e):\Sigmadecomp_\F(X))}\prod_{(y:Y)}\cyc(Z_y)
  \end{equation*}
  Furthermore, via \cref{finitely-indexed-sigma-decomposition} we obtain an equivalence
  \begin{equation*}
    \cycledecomposition(X)\simeq\sum_{(R:\EqReld(X))}\prod_{(y:X/R)}\cyc(\fib_{q_R}(y))
  \end{equation*}
  This equivalence already makes use of the univalence axiom. By the above equivalences it suffices to construct an equivalence
  \begin{equation}\label{eq:unique-cycle-decomposition}
    \left(\sum_{(R:\EqReld(X))}\prod_{(y:X/R)}\cyc(\fib_{q_R}(y))\right)\simeq (X\simeq X).\tag{\textasteriskcentered}
  \end{equation}

  To construct the underlying map of this equivalence, consider a decidable equivalence relation $R$ on $X$. Let us write $F_y\defeq\fib_{q_R}(y)$ for any $y:X/R$, and consider a cyclic structure $s_y:F_y\to F_y$ on each $F_y$. Then each $s_y$ is an equivalence, so it follows that the map
  \begin{equation*}
    {\textstyle \tot(s):\Big(\sum_{(y:X/R)}F_y\Big)\to\Big(\sum_{(y:X/R)}F_y\Big)}
  \end{equation*}
  given by $\tot(s)(y,z)\defeq (y,s_y(z))$ is an equivalence. Since there is a canonical equivalence $e:X\simeq \sum_{(y:X/R)}F_y$, we obtain a unique equivalence $f_{R,s}:X\simeq X$ such that
  \begin{equation*}
    e(f_{R,s}(x))=\tot(s)(e(x)).
  \end{equation*}
  This defines a map
  \begin{equation*}
    {\psi:\left(\sum_{(R:\EqReld(X))}\prod_{(y:X/R)}\cyc(\fib_{q_R}(y))\right)\to (X\simeq X)}.
  \end{equation*}

  In order to show that the map $\psi$ is an equivalence, we show that for every equivalence $f:X\simeq X$ there is a unique equivalence relation $R$ equipped with a cyclic structure $s_y$ on each of its equivalence classes $\fib_{q_R}(y)$ such that $\psi(R,s)=f$.
  
  Given an equivalence $f:X\simeq X$, we define the equivalence relation $R_f$ induced by
  \begin{equation*}
    R_f(x,y)\defeq \exists_{(n:\N)}f^n(x)=y.
  \end{equation*}
  In other words, $R_f$ is the reflexive and transitive closure of the relation $x,y\mapsto f(x)=y$. Note that $R_f$ is symmetric because $X$ is assumed to be finite. We will write $q_f:X\to X/R_{f}$ for the quotient map.

  Next, we need to construct a cyclic structure on the equivalence classes $F_y\defeq \fib_{q_f}(y)$, for each $y:X/R_f$. Note that $f:X\simeq X$ restricts to an equivalence $s_{f,y}:F_y\simeq F_y$ because $R_f(x,f(x))$ holds trivially for every $x:X$. It follows immediately from \cref{lemma:cyclic-test} that the proposition
  \begin{equation*}
    \brck{(\Z/n,\suc)=(F_y,s_{f,y})}
  \end{equation*}
  holds. Furthermore, since $s_{f,y}$ is defined by restricting $f$, we have
  \begin{equation*}
    e(f(x))=\tot(s_f)(e(x)).
  \end{equation*}
  We conclude that $\psi(R_f,s_f)=f$, so we have constructed a center of contraction in $\fib_{\psi}(f)$.

  To show that there is at most one element in $\fib_{\psi}(f)$, consider an equivalence relation $R$ on $X$ with equivalence classes $F_y\defeq \fib_{q_R}(y)$ for $y:X/R$. Furthermore, consider a cyclic structure $s_y:F_y\to F_y$ for each $y:X/R$ such that $\psi(R,s)=f$. Then $f$ is the unique equivalence $X\simeq X$ such that
  \begin{equation*}
    e(f(x))=\tot(s)(e(x)),
  \end{equation*}
  where $e:X\simeq\sum_{(y:X/R)}F_y$. Since $s$ is a family cyclic structure on the equivalence classes of $R$, it follows from \cref{lemma:cyclic-test} and the unique characterization of $f$, that $R(x,y)$ holds if and only if there exists an $n:\N$ such that $f^n(x)=y$. This implies that $R=R_f$.

  To finish off the proof, we will use a \emph{recalibration tactic}: It suffices to show that for any cyclic structure $s$ on the equivalence classes of $R_f$ such that $\psi(R_f,s)=f$ holds, we must have $s=s_f$. This recalibration of $s$ from $R$ to $R_f$ allows us to bypass a complicated application of transport. See \cref{remark:recalibration} below for some further remarks.

  It remains to show that if we have a family of cyclic structures $s_y:F_y\to F_y$ on the equivalence classes of $R_f$ such that $\psi(R_f,s)=f$, then $s_y=s_{f,y}$ for each $y:X/R_f$. Since
  \begin{equation*}
    e(f(x))=\tot(s)(e(x)) \qquad\text{and}\qquad e(f(x))=\tot(s_f(e(x))
  \end{equation*}
  it follows that $\tot(s)=\tot(s_f)$. This implies that the restriction of $\tot(s)$ to $F_y$ is equal to the restriction of $\tot(s_f)$ to $F_y$. In other words, $s_y=s_{f,y}$.
\end{proof}

We finally arrive, via our univalent detour, at the classical result that the transpositions generate the symmetric groups.

\begin{proposition}
  The transpositions generate the automorphism group $X\simeq X$ of a finite set $X$.
\end{proposition}

\begin{proof}
  The claim is that the map
  \begin{equation*}
    {\textstyle \lst\binomial{X}{\Fin{2}}}\to (X\simeq X)
  \end{equation*}
  given by composing the listed transpositions is a surjective map. By \cref{theorem:cycle-decompositions} it follows that for every automorphism $f:X\simeq X$ there is a unique cycle decomposition $(Y,C,e)$ of $X$ such that
  \begin{equation*}
    e(f(x))=(y,s_y(z)),\qquad\text{where}\qquad e(x)=(y,z).
  \end{equation*}
  Since we're proving a proposition we may assume that $Y=\Fin{m}$ where $m$ is the number of cycles of the automorphism $f$. Then we have $C:\Fin{m}\to \Cycle_{\F}$. Furthermore, we may assume that $(\Z/k,\suc)=C_y$ for each $y:\Fin{m}$. Then there is for each $y:\Fin{m}$ a unique automorphism $f_y:X\simeq X$ such that
  \begin{equation*}
    e(f_y(x)) \defeq
    \begin{cases}
      (y,\suc(z)) & \text{if }e(x)=(y,z)\text{ for some }z \\
      e(x) & \text{otherwise.}
    \end{cases}
  \end{equation*}
  Since the supports of the maps $f_y$, i.e., the type of elements $x:X$ such that $f_y(x)\neq x$, are pairwise disjoint, we have
  \begin{equation*}
    f=f_0\circ\cdots\circ f_{m-1}.
  \end{equation*}

  We therefore see that it suffices to show that the successor map $\suc:\Z/k\to\Z/k$ is a composite of transpositions, for any $k\geq 1$. This is easy:
  \begin{equation*}
    \suc=\transposition{0}{1}\transposition{1}{2}\cdots\transposition{(k-2)}{(k-1)}.\qedhere
  \end{equation*}
\end{proof}

\begin{remark}\label{remark:recalibration}
  The recalibration tactic used in the proof of \cref{theorem:cycle-decompositions} can be made formally precise. Consider a type $A$ equipped with $a:A$ and a type family $B$ over $A$ equipped with $b:B(a)$. If the following two conditions hold, then the total space $\sum_{(x:A)}B(x)$ is contractible:
  \begin{enumerate}
  \item For any pair $(x,y):\sum_{(x:A)}B(x)$ we have $a=x$.
  \item For any $y:B(x)$ we have $b=y$.
  \end{enumerate}
  Indeed, the second condition implies that $B(a)$ is contractible. This implies that the first condition suffices to prove contractibility of the total space.
\end{remark}

\begin{remark}
  The idea that the type of automorphisms of a finite set $X$ is equivalent to the type of cycle decompositions of $X$ can be extended to endomorphisms of $X$. The type of maps $X\to X$ is equivalent to the type of quadruples $(Y,C,T,e)$ consisting of
  \begin{enumerate}
  \item a finite type $Y$,
  \item a family $C:Y\to \Cycle_\F$ of finite cycles over $Y$,
  \item a family $T:\prod_{(y:Y)}C(y)\to\mathsf{Tree}_\F$ of finite (rooted) trees, where
    \begin{equation*}
      \mathsf{Tree}_\F\defeq\mathsf{W}_{(X:\F)}X,
    \end{equation*}
    and
  \item and equivalence
    \begin{equation*}
      e:X\simeq \sum_{(y:Y)}\sum_{(z:C(y))}T(y,z),
    \end{equation*}
    where the underlying type of a tree is defined to be its type of nodes.
  \end{enumerate}
  The reader is invited to establish this equivalence, by taking an endomorphism $f : X \to X$ on a finite type $X$ and picturing the graph with vertices X and an edge from $x$ to $f(x)$, as a collection of disjoint cycles, with a finite tree attached to each vertex of each cycle.
\end{remark}

\section{Detecting deloopings of the sign homomorphism}\label{sec:detection-sign-homomorphism}

Consider a pointed map
\begin{equation*}
  Q:\BS_n\to_\ast \BS_2
\end{equation*}
where $n\geq 2$, with $\varphi:Q_{\Fin{n}}\simeq\Fin{2}$ witnessing that $Q$ is a pointed map. In this section we give a precise characterization of when $Q$ is a delooping of the sign homomorphism. The map $Q$ is said to be a delooping of the sign homomorphism if the square
\begin{equation*}
  \begin{tikzcd}
    S_n \arrow[r,"\sign"] \arrow[d,swap,"\eqequiv"] & S_2 \arrow[d,"\eqequiv"] \\
    \loopspace\BS_n \arrow[r,swap,"\loopspace Q"] & \loopspace\BS_2
  \end{tikzcd}
\end{equation*}
commutes. Note that the case $n=1$ is handled separately: The type of pointed maps $\BS_1\to_\ast\BS_2$ is contractible, and all pointed maps in this contractible type are deloopings of the sign homomorphism.

By the equivalence $\varphi$, we obtain a base point $q:Q_{\Fin{n}}$. Note also that via univalence, we have an action on equivalences
\begin{equation*}
  e\mapsto Q_e:(X\simeq Y)\to (Q_X\simeq Q_Y),
\end{equation*}
indexed by $X,Y:\BS_n$, which is the unique family of maps for which the square
\begin{equation*}
  \begin{tikzcd}
    (X=Y) \arrow[r,"\ap_Q"] \arrow[d,swap,"\equiveq"] & (Q_X=Q_Y) \arrow[d,"\equiveq"] \\
    (X\simeq Y) \arrow[r,swap,"e\mapsto Q_e"] & (Q_X\simeq Q_Y)
  \end{tikzcd}
\end{equation*}
commutes. Furthermore, we obtain a commuting square
\begin{equation*}
  \begin{tikzcd}
  (\Fin{n}=\Fin{n}) \arrow[r,"\loopspace Q"] \arrow[d,swap,"\equiveq"] & (\Fin{2}=\Fin{2}) \arrow[d,"\simeq"] \\
  (\Fin{n}\simeq \Fin{n}) \arrow[r,swap,"e\mapsto Q_e"] & (Q_{\Fin{n}}\simeq Q_{\Fin{n}}).
  \end{tikzcd}
\end{equation*}

\begin{definition}
  We define $\sign^Q : S_n \to S_2$ to be the unique map equipped with identifications
  \begin{equation*}
    \sign^Q(e) =
    \begin{cases}
      1 & \text{if }q=Q_e(q) \\
      -1 & \text{otherwise}
    \end{cases}
  \end{equation*}
\end{definition}

\begin{lemma}\label{signsign'equality}
  If $Q_{\transposition{i}{j}}$ swaps the elements of $Q_{\Fin{n}}$ for every transposition $\transposition{i}{j}$ on $\Fin{n}$, then it follows that $\sign$ and $\sign^Q$ are equal on transpositions.
\end{lemma}

\begin{proof}
  Let $i$ and $j$ be two elements of $\Fin{n}$ such that $i < j$. Then $q\neq Q_{\transposition{i}{j}}(q)$ by assumption. This implies that $\sign^Q\transposition{i}{j}=-1=\sign\transposition{i}{j}$.
\end{proof}

\begin{lemma}\label{2ndcommutingsquare}
  Consider a permutation $e:S_n$. Then the square
  \begin{equation*}
    \begin{tikzcd}
      {Q_{\Fin{n}}} \arrow[d,swap,"\varphi"] \arrow[r,"Q_e"] & {Q_{\Fin{n}}} \arrow[d,"\varphi"] \\
      {\Fin{2}} \arrow[r,swap,"\sign^Q(e)"] & {\Fin{2}}
    \end{tikzcd}
  \end{equation*}
  commutes.
\end{lemma}

\begin{proof}
  First note that every type in the asserted square is a $2$-element type, and all the maps in the square are equivalences. Therefore it suffices to show that it commutes at one value, i.e., it suffices to show that
  \begin{equation*}
    \sign^Q(e)(\varphi(q))=\varphi(Q_e(q)).
  \end{equation*}
  Since $Q_{\Fin{n}}$ has decidable equality, we may proceed by case analysis on $q=Q_e(q)$ or $q\neq Q_e(q)$.

  If $q=Q_e(q)$, then we have the identifications $\sign^Q(e)=1$ and $\varphi(q)=\varphi(Q_e(q))$. This proves the claim in the first case.

  If $q\neq Q_e(q)$, then we have the identifications $\sign^Q(e)=-1$ and $\swap(q)=Q_e(q)$. It follows that
  \begin{equation*}
    \sign^Q(e)(\varphi(q))=\varphi(\swap(q))=\varphi(Q_e(q)).
  \end{equation*}
  This proves the claim in the second case.
\end{proof}

\begin{proposition}\label{proposition:Q-deloops-sign}
  Suppose that $Q_{\transposition{i}{j}}$ swaps the elements of $Q_{\Fin{n}}$ for every transposition $\transposition{i}{j}$ on $\Fin{n}$. Then we have a commuting square
  \begin{equation*}
    \begin{tikzcd}
      S_n \arrow[r,"\sign"] \arrow[d, "\simeq"] & S_2 \arrow[d, "\simeq"]\\
      \loopspace \BS_n \arrow[r,swap,"\loopspace\sigma"] & \loopspace \BS_2
    \end{tikzcd}
  \end{equation*}
  It follows that $\sign^Q=\sign$.
\end{proposition}

\begin{proof}
  It suffices to show that the square
  \begin{equation*}
    \begin{tikzcd}
      S_n \arrow[r,"\sign^Q"] \arrow[d,swap,"\simeq"] & S_2 \arrow[d, "\simeq"]\\
      \loopspace \BS_n \arrow[r,swap,"\loopspace Q"] & \loopspace \BS_2
    \end{tikzcd}
  \end{equation*}
  commutes. Indeed, if this square commutes, then it follows that $\sign^Q$ is a group homomorphism. Since the permutation group $S_n$ is generated by the transpositions, and since $\sign\transposition{i}{j}=\sign^Q\transposition{i}{j}$ for every transposition $\transposition{i}{j}$ on $\Fin{n}$, it follows that $\sign$ and $\sign^Q$ are equal group homomorphisms.
  
  The square with $\sign^Q$ commutes if and only if the triangle
  \begin{equation*}
    \begin{tikzcd}[column sep=large]
      S_n \arrow[r,"e\mapsto Q_e"] \arrow[dr,swap,"\sign^Q"] & (Q_{\Fin{n}}\simeq Q_{\Fin{n}}) \arrow[d,"\simeq"] \\
      & S_2
    \end{tikzcd}
  \end{equation*}
  commutes. This follows directly from \cref{2ndcommutingsquare}.
\end{proof}

We conclude that $\sigma$ is a delooping of the sign homomorphism.

\begin{theorem}\label{theorem:recognition-delooping-sign-homomorphism}
  Consider a pointed map $Q:\BS_n\to_\ast \BS_2$, where $n\geq 2$. The following are equivalent:
  \begin{enumerate}
  \item The total space
    \begin{equation*}
      \sum_{(X:\BS_n)}Q_X
    \end{equation*}
    is connected.
  \item The map $Q$ is connected.
  \item The map
    \begin{equation*}
      e\mapsto Q_e:S_n\to (Q_{\Fin{n}}\simeq Q_{\Fin{n}})
    \end{equation*}
    is surjective.
  \item The equivalence $Q_{\transposition{i}{j}}:Q_{\Fin{n}}\simeq Q_{\Fin{n}}$ swaps the elements of $Q_{\Fin{n}}$ for any transposition $\transposition{i}{j}$ on $\Fin{n}$.
  \item The square
    \begin{equation*}
      \begin{tikzcd}
        S_n \arrow[r,"\sign"] \arrow[d,swap,"\simeq"] & S_2 \arrow[d, "\simeq"]\\
        \loopspace \BS_n \arrow[r,swap,"\loopspace Q"] & \loopspace \BS_2
      \end{tikzcd}
    \end{equation*}
    commutes.
  \end{enumerate}
\end{theorem}

\begin{proof}
  The fact that 1) is equivalent to 2) follows from the fact that the total space of $Q$ is equivalent to the fiber of the map $Q:\BS_n\to_\ast \BS_2$, which was proven in Theorem II.7 of \cite{BR2017realprojective}.
  The fact that 2) is equivalent to 3) follows, since $Q$ is a surjective map, so it is connected if and only if $\loopspace Q$ is surjective, which is an equivalent way of phrasing 3). It is immediate that 4) implies 3). To see that 3) implies 4), recall that the transpositions generate the permutation group $S_n$. Therefore it follows from 3) that there is a transposition $\transposition{i}{j}$ such that $Q_{\transposition{i}{j}}$ swaps the elements of $Q_{\Fin{n}}$. However, if one of the transpositions swaps the elements of $Q_{\Fin{n}}$, then all of them do, since
  \begin{equation*}
    Q_{\transposition{a}{b}}=Q_{\transposition{a}{c}\transposition{b}{d}\transposition{a}{b}\transposition{b}{d}\transposition{a}{c}}=Q_{\transposition{c}{d}}.
  \end{equation*}
  The fact that 4) implies 5) was shown in \cref{proposition:Q-deloops-sign}, and the fact that 5) implies 4) is immediate.
\end{proof}

Using the existence of a delooping of the sign homomorphism we immediately obtain the following corollary.

\begin{corollary}\label{corollary:is-contractible-type-of-deloopings-sign-homomorphism}
  The type of pointed connected maps $\BS_n\to_\ast BS_2$ is contractible.
\end{corollary}

One way to construct a delooping of the sign homomorphism, is to construct a pointed map $Q:\BS_n\to_\ast \BS_2$ as a family of quotient types, starting from a type family $D:\BS_n\to\UU$.

\begin{corollary}\label{cor:quotient-construction}
  Consider a type family $D:\BS_n\to\UU$ equipped with a point $d:D_{\Fin{n}}$, and consider a family of equivalence relations $R_X$ on $D_X$ indexed by $X:\BS_n$ equipped with a pointed equivalence
  \begin{equation*}
    \varphi:D_{\Fin{n}}/R_{\Fin{n}}\simeq_\ast \Fin{2}.
  \end{equation*}
  If we have $\neg R_{\Fin{n}}(d,D_{\transposition{i}{j}}(d))$ for each transposition $\transposition{i}{j}$ on $\Fin{n}$, then the pointed map
  \begin{equation*}
    X\mapsto D_X/R_X:\BS_n\to_\ast\BS_2
  \end{equation*}
  is a delooping of the sign homomorphism.
\end{corollary}

\begin{proof}
  First, we note that the type $D_X/R_X$ is indeed a $2$-element type for each $X:\BS_n$, since being a $2$-element type is a property and we have the equivalence $\varphi$. Therefore we have a pointed map
  \begin{equation*}
    Q:\BS_n\to_\ast \BS_2
  \end{equation*}
  given by $Q_X\defeq D_X/R_X$. By the condition that $d$ and $D_{\transposition{i}{j}}(d)$ are unrelated by $R_{\Fin{n}}$, it follows that $q(d)\neq q(D_{\transposition{i}{j}}(d))$ for each transposition $\transposition{i}{j}$ on $\Fin{n}$, where $q$ is the quotient map. Note that the square
  \begin{equation*}
    \begin{tikzcd}
      D_X \arrow[r,"D_e"] \arrow[d,swap,"q"] & D_Y \arrow[d,"q"] \\
      Q_X \arrow[r,swap,"Q_e"] & Q_Y
    \end{tikzcd}
  \end{equation*}
  commutes for every equivalence $e:X\simeq Y$. Therefore it follows that $q(d)\neq Q_{\transposition{i}{j}}(q(d))$ for each transposition $\transposition{i}{j}$ on $\Fin{n}$. This implies that $Q$ is a delooping of the sign homomorphism.
\end{proof}

\section{Delooping the sign homomorphism with fixed points}\label{sec:delooping-sign-fixed-points}

\begin{definition}
  Consider an $n$-element set $X$. Then $S_n$ acts on the set
  \begin{equation*}
    A_X\defeq (\Fin{n}\simeq X)\to S_2
  \end{equation*}
  by the action $(\alpha\cdot f)(h) \defeq \sign(\alpha)\circ f(h\circ\alpha^{-1})$. Furthermore, define
  \begin{equation*}
    Q_X\defeq \mathsf{fix}(A_X)
  \end{equation*}
  to be the type of fixed points of the $S_n$ action on $A_X$.
\end{definition}

\begin{proposition}
  There is a pointed equivalence
  \begin{equation*}
    Q_{\Fin{n}}\simeq_\ast \Fin{2}.
  \end{equation*}
  In particular, the type $Q_X$ has two elements for each $X:\BS_n$.
\end{proposition}

\begin{proof}
  Suppose $f:A_{\Fin{n}}$ is a fixed point of the $S_n$-action on $A_{\Fin{n}}$. Then it follows that
  \begin{equation*}
    f(\alpha)=\sign(\alpha)\circ f(\alpha\circ\alpha^{-1})=\sign(\alpha)\circ f(1).
  \end{equation*}
  This shows that all the fixed points are of the form $\alpha\mapsto \sign(\alpha)\circ (\pm 1)$. In other words, the fixed points are $\pm\sign$. We choose $\sign$ to be the base point of $Q_{\Fin{n}}$.
\end{proof}

\begin{theorem}
  The pointed map $Q:\BS_n\to_\ast\BS_2$ is a delooping of the sign homomorphism.
\end{theorem}

\begin{proof}
  By \cref{theorem:recognition-delooping-sign-homomorphism} it suffices to check that $Q_{\transposition{i}{j}}$ swaps the elements of $Q_{\Fin{n}}$ for each transposition $\transposition{i}{j}$ on $\Fin{n}$.
  
  First note that the action on equivalences $e\mapsto A_e$ is given by $A_e(f)(h)\defeq f(e^{-1}\circ h)$. One can easily see that this action on equivalences preserves fixed points since $A_{\id}$ preserves fixed points, and that its restriction to $Q$ is the action on equivalences of $Q$. It follows that
  \begin{equation*}
    Q_{\transposition{i}{j}}(\pm\sign)(\alpha)=\pm\sign(\transposition{i}{j}\circ\alpha)=\mp\sign(\alpha).
  \end{equation*}
  In other words, $Q_{\transposition{i}{j}}$ swaps the elements of $Q_{\Fin{n}}$.
\end{proof}

\section{Delooping the sign homomorphism with orbits}\label{sec:delooping-sign-orbits}

\begin{definition}
  Consider an $n$-element type $X$. Then $S_n$ acts on the set
  \begin{equation*}
    B_X\defeq (\Fin{n}\simeq X)\times S_2
  \end{equation*}
  by the action $\alpha\cdot(h,\sigma)\defeq (h\circ\alpha^{-1},\sign(\alpha)\circ\sigma)$. We define
  \begin{equation*}
    Q_X\defeq B_X/\sim
  \end{equation*}
  where $(h,\sigma)\sim (h',\sigma')$ if $(h,\sigma)$ and $(h',\sigma')$ are in the same orbit of $S_n$.
\end{definition}

\begin{proposition}
  There is a pointed equivalence
  \begin{equation*}
    Q_{\Fin{n}}\simeq_\ast\Fin{2}.
  \end{equation*}
  In particular, the type $Q_X$ has two elements for each $X:\BS_n$.
\end{proposition}

\begin{proof}
  Note that
  \begin{equation*}
    (h,\sigma)\sim(1,\sign(h)\circ\sigma)
  \end{equation*}
  for any $(h,\sigma):B_{\Fin{2}}$. This shows that there are precisely two orbits: the orbit of $(1,1)$ and the orbit of $(1,-1)$.
\end{proof}

\begin{theorem}
  The pointed map $Q:\BS_n\to_\ast\BS_2$ is a delooping of the sign homomorphism.
\end{theorem}

\begin{proof}
    By \cref{theorem:recognition-delooping-sign-homomorphism} it suffices to check that $Q_{\transposition{i}{j}}$ swaps the elements of $Q_{\Fin{n}}$ for each transposition $\transposition{i}{j}$ on $\Fin{n}$.
  
  First note that the action on equivalences $e\mapsto B_e$ is given by $B_e(h,\sigma))\defeq (e\circ h,\sigma)$. One can easily see that this action on equivalences preserves orbits since $A_{\id}$ preserves orbits, and that its quotient to $Q$ is the action on equivalences of $Q$. It follows that
  \begin{equation*}
    Q_{\transposition{i}{j}}([1,\pm 1]_{\sim})=[\transposition{i}{j},\pm 1]_{\sim}=[1,\mp 1]_{\sim}.
  \end{equation*}
  In other words, $Q_{\transposition{i}{j}}$ swaps the elements of $Q_{\Fin{n}}$.
\end{proof}

\section{Simpson's delooping of the sign homomorphism}\label{sec:simpson}

In this section we will use \cref{cor:quotient-construction} to define a delooping of the sign homomorphism based on an suggestion of Alex Simpson. This is perhaps the simplest way to describe the delooping of the sign homomorphism. We will fix a natural number $n\geq 2$.

\begin{definition}
  We define the type family $E:\BS_n\to \UU$ by
  \begin{equation*}
    E_X := \Fin{n} \simeq X.
  \end{equation*}
  The canonical element $d:E_{\Fin{n}}$ is defined by $d\defeq\id$.
\end{definition}

The action of $E$ on equivalences is described by
\begin{equation*}
  E_e(f)\defeq e\circ f
\end{equation*}
for any equivalence $e:X\simeq Y$. 

\begin{definition}
  For any $X:\BS_n$ we define a decidable relation $R_X:E_X\to E_X\to\dProp$ by
  \begin{equation*}
    R_X(f,g):= \big(\sign(g^{-1} \circ f)= 1\big)
  \end{equation*}
\end{definition}

\begin{lemma}
  The relation $R_X$ is an equivalence relation on $E_X$, for any $X:\BS_n$.
\end{lemma}

\begin{proof}
  Reflexivity of $R_X$ follows from
  \begin{equation*}
    \sign(f^{-1}\circ f)=\sign(\id)=1.
  \end{equation*}
  To see that $R_X$ is symmetric, assume that $\sign(g^{-1}\circ f)=1$. Then we have
  \begin{equation*}
    \sign(f^{-1}\circ g)=\sign(g^{-1}\circ f)^{-1}=1.
  \end{equation*}
  To see that $R_X$ is transitive, assume that $\sign(h^{-1}\circ g)=1$ and that $\sign(g^{-1}\circ f)=1$. Then we have
  \begin{equation*}
    \sign(h^{-1}\circ f)=\sign(h^{-1}\circ g)\circ\sign(g^{-1}\circ f)=1.\qedhere
  \end{equation*}
\end{proof}

\begin{lemma}\label{atleasttwoecsimpson}
  For any transposition $\transposition{i}{j}$ on $\Fin{n}$ we have
  \begin{equation*}
    \neg R_{\Fin{n}}(d,E_{\transposition{i}{j}}(d)).
  \end{equation*}
\end{lemma}

\begin{proof}
  By the definition of the sign homomorphism we have $\sign\transposition{i}{j}=-1$. Therefore it follows that
  \begin{equation*}
    \sign(E_{\transposition{i}{j}}(d)^{-1}\circ d)=\sign\transposition{i}{j}=-1\neq 1.\qedhere
  \end{equation*}
\end{proof}

\begin{lemma}\label{equivalencefin2simpson}
  There is a pointed equivalence
  \begin{equation*}
    \varphi:E_{\Fin{n}}/R_{\Fin{n}}\simeq_\ast \Fin{2}.
  \end{equation*}
\end{lemma}

\begin{proof}
  It suffices to show that $R_{\Fin{n}}$ has exactly two equivalence classes. It follows from \cref{atleasttwoecsimpson} that there are at least two equivalence classes, so it remains to show that there are at most two equivalence classes. Consider $f,f':E_{\Fin{n}}$ such that $\neg R_{\Fin{n}}(f,f')$, and consider $g:E_{\Fin{n}}$. Then it follows that
  \begin{equation*}
    \sign(f'^{-1}\circ g)\circ\sign(g^{-1}\circ f)=\sign(f'^{-1}\circ f)=-1.
  \end{equation*}
  This implies that $\sign(f'^{-1}\circ g)$ and $\sign(g^{-1}\circ f)$ can't both be $-1$, so one of them is $1$.
\end{proof}

\begin{theorem}\label{theorem:simpson}
  The \define{Simpson map} $Q:\BS_n\to_\ast \BS_2$ given by
  \begin{equation*}
    Q_X\defeq E_X/R_X
  \end{equation*}
  is a delooping of the sign homomorphism.
\end{theorem}

\begin{proof}
  In \cref{atleasttwoecsimpson} we showed that
  \begin{equation*}
    \neg R_{\Fin{n}}(d,E_{\transposition{i}{j}}(d))
  \end{equation*}
  holds for every transposition $\transposition{i}{j}$ on $\Fin{n}$, and in \cref{equivalencefin2simpson} we constructed a pointed equivalence
  \begin{equation*}
    \varphi:E_{\Fin{n}}/R_{\Fin{n}}\simeq_\ast\Fin{2}.
  \end{equation*}
  Therefore the claim follows from \cref{cor:quotient-construction}.
\end{proof}

\section{Cartier's delooping of the sign homomorphism}\label{sec:cartier}

In this section we will again use \cref{cor:quotient-construction} to give an alternative definition of a delooping of the sign homomorphism based on a method of Cartier~\cite{cartier1970remarques}. Unlike the previous construction, the definition of the \emph{Cartier map}
\begin{equation*}
  C:\BS_n\to_\ast \BS_2
\end{equation*}
given in the statement of \cref{theorem:cartier-delooping} does not refer to the sign homomorphism. Of course, we still refer to the sign homomorphism in order to show that $C$ is a delooping of the sign homomorphism.

Recall that the edges of the complete undirected graph on a finite set $X$ can be described equivalently as $2$-element (decidable) subsets of $X$, where we also note that any finite subset of a finite type is automatically decidable. The main idea behind Cartier's delooping of the sign homomorphism is that there is a natural equivalence relation on the type of \emph{orientations} of the complete undirected graph on a finite set $X$, of which the quotient is a $2$-element set. The operation that sends an $n$-element set $X$ to this quotient deloops the sign homomorphism.

\begin{definition}
  We define the type $D_X$ of {\bf orientations} of the complete undirected graph on $X$ to be the type of functions which pick an element of each $2$-element decidable subset of $X$, i.e.,
  \begin{equation*}
    D_X := \prod_{P:\binomial{X}{\Fin{2}}}\sum_{x:X}P(x).
  \end{equation*}
  The canonical orientation $d:D_{\Fin{2}}$ is defined by
  \begin{equation*}
    d(P)\defeq\max(P).
  \end{equation*}
\end{definition}

The action of $D$ on equivalences is described by
\begin{equation*}
  D_e(u,P)\defeq e(u(P\circ e)).
\end{equation*}
Indeed, if $e:X\simeq Y$ is an equivalence, $u:D_X$ is an orientation of the complete undirected graph of $X$, and $P:Y\to\dProp$ is a $2$-element decidable subset of $Y$, then $u(P\circ e)$ is an element of the $2$-element decidable subset $P\circ e$ of $X$, and furthermore the equivalence $e$ restricts to an equivalence $\sum_{(x:X)}P(e(x))\simeq\sum_{(y:Y)}P(y)$. To see that $D_e$ indeed describes the action of $D$ on equivalences, we have to show that the square
\begin{equation*}
  \begin{tikzcd}
    (X=Y) \arrow[r] \arrow[d]  & (D_X=D_Y) \arrow[d] \\
    (X\simeq Y) \arrow[r] & (D_X\simeq D_Y)
  \end{tikzcd}
\end{equation*}
commutes for every identification $p:X=Y$. This is shown by path induction on $p$. Indeed it is the case that $D_\id(u,P)=u(P)$. In other words, $D_\id=\id$.

\begin{definition}
  We define a decidable relation $R_X$ on $D_X$ in two steps:
  \begin{enumerate}
  \item For any two elements $u$ and $v$ in $D_X$ we define the natural number $m(u,v)$ of \textbf{relative inversions} to be the number of $2$-element decidable subsets on which $u$ and $v$ differ.
  \item For any two elements $u$ and $v$ in $D_X$ we define
    \begin{equation*}
      R_X(u,v)\defeq (m(u,v)\equiv 0\mod 2).
    \end{equation*}
  \end{enumerate}
\end{definition}

\begin{lemma}\label{lemma:unrelated-transposition-max}
  For any transposition $\transposition{i}{j}$ on $\Fin{n}$, the orientation $D_{\transposition{i}{j}}(d)$ is not related by $R_{\Fin{n}}$ to the canonical orientation $d$.
\end{lemma}

\begin{proof}
  Define $F$ to be the type of $2$-element decidable subsets of $\Fin{n}$ at which $d$ and $D_{\transposition{i}{j}}(d)$ differ. Our goal is to show that the cardinality of $F$ is odd.
  
  Note that the transposition $\transposition{i}{j}$ induces a $\Z/2$-action on the set of $2$-element decidable subsets of $\Fin{n}$, given by
  \begin{equation*}
    P\mapsto P\circ\transposition{i}{j}.
  \end{equation*}
  Now we observe that
  \begin{equation*}
    \big(\transposition{i}{j}(d(P\circ\transposition{i}{j}))=d(P)\big)
    \Leftrightarrow
    \big(\transposition{i}{j}(d(P))=d(P\circ\transposition{i}{j})\big)
  \end{equation*}
  for any $2$-element decidable subset $P$ of $\Fin{n}$. This implies that if $P\in F$, then the entire orbit of $P$ is contained in $F$. The orbits containing two elements contribute an even number of elements to $F$. Therefore it suffices to show that $F$ contains exactly one fixed point. The fixed points are the $2$-element subsets disjoint from $\{i,j\}$ and the subset $\{i,j\}$ itself. If $P$ is disjoint from $\{i,j\}$, then
  \begin{equation*}
    D_{\transposition{i}{j}}(d,P)=\transposition{i}{j}(d(P\circ\transposition{i}{j}))=\transposition{i}{j}(d(P))=d(P).
  \end{equation*}
  Such subsets $P$ are clearly not in $F$. On the other hand, if $P=\{i,j\}$ then we find that
  \begin{equation*}
    D_{\transposition{i}{j}}(d,P)=\transposition{i}{j}(d(P\circ\transposition{i}{j}))=\transposition{i}{j}(d(P))=\min(P).
  \end{equation*}
  Thus we see that $\{i,j\}\in F$, completing the proof.
\end{proof}

In order to show that $R_X$ is an equivalence relation, we prove the following lemma.

\begin{lemma}\label{symmetricdifference}
  For any three elements $u_1$, $u_2$ and $u_3$ of type $D_X$, we have
  \begin{equation*}
    m(u_1,u_3) \equiv m(u_1,u_2) + m(u_2,u_3) \mod 2.
  \end{equation*}
\end{lemma}

\begin{proof}
  For each $i,j\in\{1,2,3\}$ we define
  \begin{equation*}
    F_{i,j} := \{P\mid u_i(P) \neq u_j(P)\}\subseteq {\textstyle \binomial{X}{\Fin{2}}}.
  \end{equation*}
  Then $m(u_i,u_j)$ is the cardinality of $F_{i,j}$. Furthermore, let $G := F_{1,2} \cap F_{2,3}$. We claim that
  \begin{equation*}
    F_{1,3}=(F_{1,2}\setminus F_{2,3})\cup (F_{2,3}\setminus F_{1,2})=(F_{1,2}\cup F_{2,3})\setminus G.
  \end{equation*}
  \begin{equation*}
  \begin{tikzpicture}
    \def\radius{2cm}
    \def\mycolorbox#1{\textcolor{#1}{\rule{2ex}{2ex}}}
    \colorlet{colori}{blue!70}
    \colorlet{colorii}{red!70}

    \coordinate (ceni);
    \coordinate[xshift=\radius] (cenii);

    \draw[fill=colori,fill opacity=0.5] (ceni) circle (\radius);
    \draw[fill=colori,fill opacity=0.5] (cenii) circle (\radius);

    \begin{scope}[white, opacity=1]
      \clip (ceni) circle (\radius);
      \fill (cenii) circle (\radius);
    \end{scope}

    \begin{scope}[colorii, opacity=0.5]
      \clip (ceni) circle (\radius);
      \fill (cenii) circle (\radius);
    \end{scope}

    \draw (ceni) circle (\radius);
    \draw (cenii) circle (\radius);

    \draw  ([xshift=-6pt,yshift=6pt]current bounding box.north west)
      rectangle ([xshift=6pt,yshift=-6pt]current bounding box.south east);

    \node[yshift=10pt] at (current bounding box.north) {$2$-element decidable subsets of $X$};

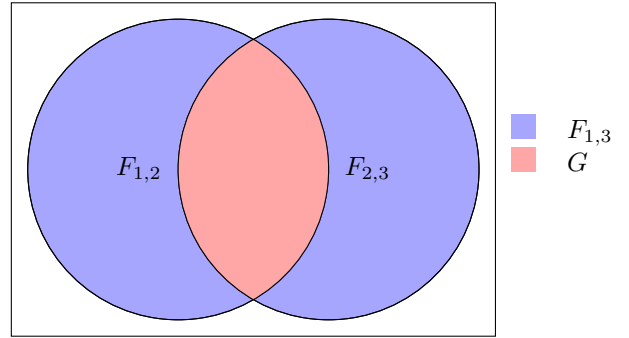
    \node at ([xshift=0.44*\radius]current bounding box.east)
    {
    \begin{tabular}{ll}
      \mycolorbox{colori!50} & $F_{1,3}$  \\
      \mycolorbox{colorii!50} & $G$
    \end{tabular}
    };

    \node[xshift=-.5\radius] at (ceni) {$F_{1,2}$};
    \node[xshift=.5\radius] at (cenii) {$F_{2,3}$};
  \end{tikzpicture}
  \end{equation*}

  To see that $F_{1,3}$ is indeed the symmetric difference of $F_{1,2}$ and $F_{2,3}$, let $P$ be a $2$-element decidable subset of $X$. If $P\in F_{1,3}$, i.e., if $u_1(P)\neq u_3(P)$, then it follows that the propositions $u_1(P)=u_2(P)$ and $u_2(P)=u_3(P)$ cannot both be false because $P$ is a $2$-element set, and they cannot both be true because $u_1(P)\neq u_3(P)$ by assumption. This proves that $F_{1,3}$ is contained in $(F_{1,2}\cup F_{2,3})\setminus G$. For the converse we have two cases:
  \begin{enumerate}
  \item If $P\in F_{1,2}\setminus F_{2,3}$, then we have $u_1(P)\neq u_2(P)$ and $u_2(P)=u_3(P)$. It follows that $u_1(P)\neq u_3(P)$.
  \item If $P\in F_{2,3}\setminus F_{1,2}$, then we have $u_1(P)=u_2(P)$ and $u_2(P)\neq u_3(P)$. It follows that $u_1(P)\neq u_3(P)$.
  \end{enumerate}
  This proves the claim.

  Since $F_{1,3}$ is the symmetric difference of $F_{1,2}$ and $F_{2,3}$, it follows that
  \begin{equation*}
    \Card{F_{1,3}} = \Card{F_{1,2}} + \Card{F_{2,3}} - 2 \times \Card{G}.
  \end{equation*}
  We conclude that $m(u_1,u_3)$ has the same parity as $m(u_1,u_2) + m(u_2,u_3)$.
\end{proof}

\begin{corollary}
  The relation $R_X$ is an equivalence relation on $D_X$.
\end{corollary}

\begin{proof}
  It is immediate that $R_X$ is reflexive and symmetric. To see that $R_X$ is transitive, let $u_1$, $u_2$ and $u_3$ be elements of $D_X$ such that $R_X(u_1,u_2)$ and $R_X(u_2,u_3)$. By \cref{symmetricdifference} it follows that $m(u_1,u_3) \equiv 0 \mod 2$, so $R_X(u_1,u_3)$ and $R_X$ is transitive.
\end{proof}

\begin{corollary}\label{corollary:pointed-equivalence-quotient}
  There is a pointed equivalence
  \begin{equation*}
    \varphi:D_{\Fin{n}}/R_{\Fin{n}}\simeq_\ast \Fin{2}.
  \end{equation*}
\end{corollary}

\begin{proof}
  It suffices to show that there are exactly two equivalence classes for $R_{\Fin{n}}$. By \cref{lemma:unrelated-transposition-max} it follows that there are at least two equivalence classes. To see that there are at most two distinct equivalence classes of $R_{\Fin{n}}$, let $u$ and $u'$ be orientations of the complete graph on $X$ such that $\neg R_{\Fin{n}}(u,u')$, and let $v$ be a third orientation. Then we have
  \begin{equation*}
    m(u,v)+m(v,u')\equiv 1\mod 2.
  \end{equation*}
  This implies that either $m(u,v)\equiv 0$ and $m(v,u')\equiv 1\mod 2$ or that $m(u,v)\equiv 1$ and $m(v,u')\equiv 0\mod 2$. It follows that either $v$ is equivalent to $u$ or that $v$ is equivalent to $u'$.
\end{proof}

\begin{theorem}\label{theorem:cartier-delooping}
  The \textbf{Cartier map} $C:\BS_n\to_\ast \BS_2$ defined by
  \begin{equation*}
    C_X\defeq D_X/R_X
  \end{equation*}
  is a delooping of the sign homomorphism.
\end{theorem}

\begin{proof}
  In \cref{lemma:unrelated-transposition-max} we showed that
  \begin{equation*}
    \neg R_{\Fin{n}}(d,D_{\transposition{i}{j}}(d))
  \end{equation*}
  holds for every transposition $\transposition{i}{j}$ on $\Fin{n}$, and in \cref{corollary:pointed-equivalence-quotient} we constructed a pointed equivalence
  \begin{equation*}
    \varphi:D_{\Fin{n}}/R_{\Fin{n}}\simeq_\ast\Fin{2}.
  \end{equation*}
  Therefore the claim follows from \cref{cor:quotient-construction}.
\end{proof}

\section{The concrete alternating groups}\label{sec:alternating-groups}

We conclude our paper by constructing the concrete alternating groups $A_n$ for each $n\geq 2$. Let
\begin{equation*}
  Q:\BS_n\to_\ast\BS_2
\end{equation*}
be a delooping of the sign homomorphism. Recall that we showed in \cref{corollary:is-contractible-type-of-deloopings-sign-homomorphism} that the type of deloopings of the sign homomorphism is contractible, so it is irrelevant which construction of the sign homomorphism we use.

\begin{definition}
  Consider a natural number $n\geq 2$. We define the \textbf{concrete alternating group} $\BA_n$ by
  \begin{equation*}
    \BA_n\defeq\sum_{X:BS_n}Q_X.
  \end{equation*}
  We will write $A_n\defeq\loopspace\BA_n$. 
\end{definition}

The following theorem justifies our definition:

\begin{theorem}
  The concrete alternating group $BA_n$ is a pointed connected $1$-type that fits in a pullback square
  \begin{equation*}
    \begin{tikzcd}
      \BA_n \arrow[r] \arrow[d,swap,"\pr_1"] & \unit \arrow[d] \\
      \BS_n \arrow[r,swap,"Q"] & \BS_2
    \end{tikzcd}
  \end{equation*}
\end{theorem}

\begin{proof}
  The fact that $\BA_n$ is pointed follows directly from the fact that $Q$ is pointed. Since $\BA_n$ is a dependent sum of $1$-types, it follows that $\BA_n$ is a $1$-type. The fact that $\BA_n$ is connected follows from \cref{theorem:recognition-delooping-sign-homomorphism}. The last claim is a special case of Theorem II.7 of \cite{BR2017realprojective}.
\end{proof}

\bibliographystyle{IEEEtran}
\bibliography{bibliography}

\end{document}